\documentclass[12pt]{amsart}
\usepackage[mathscr]{eucal}
\usepackage{amssymb}
\usepackage{latexsym}
\usepackage{amsthm}
\usepackage{amsmath}
\usepackage{graphicx}

\theoremstyle{plain}

\newtheorem{thm}{Theorem}[section]

\theoremstyle{definition}
\newtheorem{rem}{Remark}[section]
\newtheorem{dfn}{Definition}[section]

\title{Density measures on a certain flow}
\author{Ryoichi Kunisada}

\address{Department of Mathematical Science, Graduate School of Science and Engineering, Waseda University, Shinjuku-ku, Tokyo 169-8555, Japan}

\email{tk-waseda@ruri.waseda.jp}

\date{}

\begin{document}
\maketitle

\begin{abstract}
We study finitely additive extensions of the asymptotic density to all the subsets of natural numbers. Such measures are called density measures. We consider a class of density measures constructed from free ultrafilters on $\mathbb{N}$ and investigate absolute continuity and singularity for those density measures. In particular, for any pair of such density measures we prove necessary and sufficient conditions that one is absolutely continuous with respect to the other and that they are singular. Also we prove the same results for weak absolute continuity and strong singularity.
\end{abstract}

\bigskip

\section{Introduction}
We denote the set of natural numbers by $\mathbb{N}$, and the family of all subsets of $\mathbb{N}$ by $\mathcal{P}(\mathbb{N})$. Recall that the asymptotic density of a set $A \in \mathcal{P}(\mathbb{N})$ is defined as
\[D(A) = \lim_n \frac{|A \cap n|}{n} \]
if this limit exits, where $|A \cap n| = |A \cap \{1,2, \cdots, n\}|$ for any $n \in \mathbb{N}$. A finitely additive measure defined on $\mathcal{P}(\mathbb{N})$ extending the asymptotic density is called a density measure, and let us denote the set of all density measures by $\mathcal{C}$. 
In general, if a finitely additive probability measure space $(\mathbb{N}, \mathcal{P}(\mathbb{N}), \mu)$ are given, one can obtain a normalized positive linear functional $\varphi$ on $l^{\infty}$ of the Banach space of all real-valued bounded functions on $\mathbb{N}$ by the integral with respect to $\mu$. Oppositely, If a normalized positive linear functional $\varphi$ on $l^{\infty}$ is given, we get a finitely additive probability measure $\mu$ on $\mathcal{P}(\mathbb{N})$ by $\mu(A) = \varphi(I_A)$, where $I_A$ is the characteristic function of $A \in \mathcal{P}(\mathbb{N})$. In what follows, we will identify these two notions accordingly.
 
  It is known that the linear functionals corresponding to the density measures are precisely the linear functionals that extend Ces\`{a}ro mean ([7, Theorem 3.3]). Namely, it consists of those functionals $\varphi$ for which 
\[\varphi(f) = \lim_n \frac{1}{n} \sum_{i=1}^n f(i) \]
holds if this limit exists. We will deal with a certain class of such linear functionals satisfying the following condition:
\[\varphi(f) \le \limsup_{n \to \infty} \frac{1}{n} \sum_{i=1}^{n} f(i)  \]
for every $f \in l^{\infty}$. We denote by $\mathcal{C}_0$ the set of all density measures with this property. It is known that $\mathcal{C}_0$ is properly contained in $\mathcal{C}$. In particular, we consider density measures in $\mathcal{C}_0$ expressed by
\[\varphi^{\mathcal{U}}(f) = \mathcal{U}-\lim_{n} \frac{1}{n} \sum_{i=1}^n f(i), \]
where $\mathcal{U}$ is any free ultrafilter on $\mathbb{N}$ and the limit above is the limit of $f$ along an ultrafilter $\mathcal{U}$ (See [5] for more details of this notion). We denotes by $\tilde{\mathcal{C}_0}$ the set of all such density measures.  In [6] we have investigated the space $\tilde{\mathcal{C}_0}$ and proved that it is a weak-$^*$ compact subset of $\mathcal{C}_0$ and $ex(\mathcal{C}_0) \subset \tilde{\mathcal{C}_0}$ holds [6, Theorem 2.1], where $ex(\mathcal{C}_0)$ denotes the extreme points of $\mathcal{C}_0$. This means that by the Krein-Milman theorem each density measure $\varphi$ in $\mathcal{C}_0$ can be expressed in the form
\[\varphi(f) = \int_{\tilde{\mathcal{C}_0}} \phi(f)d\mu(\phi), \quad f \in l^{\infty}\]
for some probability measure $\mu$ on $\tilde{\mathcal{C}_0}$. Although we are mainly interested in density measures in $\tilde{\mathcal{C}_0}$, we will deal with this general form of density measures for a certain class of probability measures $\mu$ on $\tilde{\mathcal{C}_0}$ in section 4.

 Also we have shown that  $\tilde{\mathcal{C}_0}$ is homeomorphic to a space $\Omega^*$ on which a continuous flow induced by usual addition on the real line $\mathbb{R}$ can be defined in a natural way. This flow plays an very important role in studying density measures in $\tilde{\mathcal{C}_0}$ throughout the paper.

 The paper is organized as follows: Section 2 contains an accurate definition of the space $\Omega^*$ and the flow on it and we establish the correspondence of a element of $\Omega^*$ to a density measure in $\tilde{\mathcal{C}_0}$ by giving a explicit formula in Theorem 2.5. Also we introduce the two types of notions of absolutely continuity and singularity respectively. These notions for density measures in $\tilde{\mathcal{C}_0}$ are studied in detail in section 3. Section 4 deals with several applications of the results of section 3. In particular, the additive property of density measures in $\mathcal{C}_0$ will be studied.
\medskip

\section{Preliminaries}
 Following [3, Chapter 6] we introduce the notions of absolute continuity and singularity for finitely additive measures. From now on a $measure$ will mean a finitely additive probability measure on $\mathcal{P}(\mathbb{N})$.

In the following, let $\mu$ and $\nu$ be any two measures. 

\begin{dfn}
We say that $\nu$ is absolutely continuous with respect to $\mu$ if for any $\varepsilon > 0$, there exists $\delta > 0$ such that $\nu(A) < \varepsilon$ whenever $\mu(A) < \delta$, where $A \in \mathcal{P}(\mathbb{N})$. In this case, we write $\nu \ll \mu$.
\end{dfn}

We can consider a weak version of absolute continuity in a natural way as follows:

\begin{dfn}
We say that $\nu$ is weakly absolutely continuous with respect to $\mu$ if $\nu(A) = 0$ whenever $\mu(A) = 0$, where $A \in \mathcal{P}(\mathbb{N})$. In this case, we write $\nu \prec \mu$.
\end{dfn}

Next we define the notion of singularity.

\begin{dfn}
We say that $\mu$ and $\nu$ are singular if for every $\varepsilon > 0$, there exists a set $D \in \mathcal{P}(\mathbb{N})$ such that $\mu(D) < \varepsilon$ and $\nu(D^c) < \varepsilon$.
In this case, we write $\mu \perp \nu$.
\end{dfn}

Also we can define a strong version of singularity in a way that seemed natural.

\begin{dfn}
We say that $\mu$ and $\nu$ are strongly singular if there exists a set $D \in \mathcal{P}(\mathbb{N})$ such that $\mu(D) = 0$ and $\nu(D^c) = 0$. In this case, we write $\mu \ \rotatebox{90}{$\vDash$} \ \nu$.
\end{dfn}

These notions can be formulated in the context of the Stone-\v{C}ech compactification $\beta\mathbb{N}$ of $\mathbb{N}$. Let $\mu$ be a measure. Observe that one can identify each element $A$ of $\mathcal{P}(\mathbb{N})$ with a clopen subset $\overline{A}$ in $\beta\mathbb{N}$, where $\overline{A}$ denotes the closure of $A$ in $\beta\mathbb{N}$. Let us denote $\overline{\mathcal{P}(\mathbb{N})} = \{\overline{A} : A \in \mathcal{P} \}$ and define a measure $\hat{\mu}$ on it by $\hat{\mu}(\overline{A}) = \mu(A)$.
Since any union of disjoint family of clopen subsets can not be a clopen subset, $\hat{\mu}$ is countably additive on $\overline{\mathcal{P}(\mathbb{N})}$. Hence we can extend it to a countable additive measure on the $\sigma$-algebra generated by $\overline{\mathcal{P}(\mathbb{N})}$, that is, the Baire $\sigma$-algebra of $\beta\mathbb{N}$. This can also be extended uniquely to the Borel $\sigma$-algebra $\mathcal{B}(\beta\mathbb{N})$ of $\beta\mathbb{N}$ as a countable additive measure. We still denote it by $\hat{\mu}$. We denote by supp $\mu$ the support of $\hat{\mu}$ in $\beta\mathbb{N}$.

 The following results are easy to prove, and we omit proofs.

\begin{thm}
$\nu \ll \mu$ if and only if $\hat{\nu} \ll \hat{\mu}$, where the latter represents the existing notion of absolute continuity for countably additive measures.
\end{thm}

\begin{thm}
$\nu \prec \mu$ if and only if supp $\nu$ $\subseteq$ supp $\mu$.
\end{thm}

\begin{thm}
$\nu \perp \mu$ if and only if $\hat{\nu} \perp \hat{\mu}$, where the latter represents the existing notion of singularity for countably additive measures.
\end{thm}

\begin{thm}
$\mu \ \rotatebox{90}{$\vDash$} \ \nu$ if and only if supp $\nu$ $\cap$ supp $\mu$ = $\emptyset$.
\end{thm}

 Next we define a compact space $\Omega^*$ and a continuous flow on it. At first, we shall construct a homeomorphism on $\beta\mathbb{N} \setminus \mathbb{N}$, here denoted by $\mathbb{N}^*$, onto itself as the following. Let $\tau_0$ be the right translation on $\mathbb{N}$, i.e., $\tau_0(n) = n+1$ for each $n \in \mathbb{N}$. Then we extend it continuously to $\tau : \beta\mathbb{N} \rightarrow \beta\mathbb{N}$. It is easy to see that the restriction of $\tau$ to $\mathbb{N}^*$ is a homeomorphism onto $\mathbb{N}^*$. Next we extend $\tau$ to a continuous flow. Construct the space $\Omega^*$ from the product space $\mathbb{N}^* \times [0,1]$ by identifying the pairs of points $(\omega, 1)$ and $(\tau \omega, 0)$ for all $\omega \in \mathbb{N}^*$, and define for each $s \in \mathbb{R}$ the homeomorphism $\phi^s : \Omega^* \rightarrow \Omega^*$ by
\[\phi^s(\omega, t) = (\tau^{[t+s]}\omega, t+s-[t+s]), \quad  (\omega \in \mathbb{N}^*, t \in [0,1]), \]
where $[t+s]$ denotes the largest integer not exceeding $t+s$. We use the following notations for the orbit of $\eta$ in $\Omega^*$ for $\phi^s$: $o_+(\eta) = \{\phi^s(\eta) : s \ge 0\}, \ o_-(\eta) = \{\phi^{-s}(\eta) : s \ge 0\}$
, $o(\eta) = \{\phi^s(\eta) : s \in \mathbb{R}\}$. Furthermore we denote the closures in $\Omega^*$ of these orbits by $\overline{o}_+(\eta), \overline{o}_-(\eta), \overline{o}(\eta)$, respectively. Also we use similar notations for the orbit of $\omega$ in $\mathbb{N}^*$ for $\tau$:  $o_+(\omega) = \{\tau^n \omega : n =0,1,2, \cdots.\}, \ o_-(\omega) = \{\tau^{-n} \omega : n=0,1,2, \cdots.\}, \ o(\omega) = \{\tau^n \omega : n \in \mathbb{Z} \}$, and also $\overline{o}_+(\omega), \overline{o}_-(\omega), \overline{o}(\omega)$ represents their closures in $\mathbb{N}^*$ respectively. 

\medskip
Arguments after section 3 is developed by using the following result which gives a certain expression to the elements of $\tilde{\mathcal{C}_0}$. See [6] for a proof.

\begin{thm}
Let $\Omega^*$ be as above, and $\tilde{\mathcal{C}_0}$ be equipped with the weak-$^*$ topology. We define the mapping $\Phi: \Omega^* \rightarrow \tilde{\mathcal{C}_0}$ as follows: Let $\eta = (\omega, t) \in \Omega^*$ and $\theta = 2^t$, and let us denote the density measure $\Phi(\eta)$ by $\nu_{\eta}$, then we set
\[\nu_{\eta}(A) = \omega-\lim_n \frac{|A \cap [\theta \cdot 2^n]|}{\theta \cdot 2^n}, \quad A \in \mathcal{P}(\mathbb{N}). \]
Then $\Phi$ is a homeomorphism.
\end{thm}

 For simplicity, in the case of $t=0$, we write $\nu_{\omega}$.
\medskip

\section{Absolute continuity and singularity}
In this section we shall prove the following results:

\begin{thm}
For any two elements in $\tilde{\mathcal{C}_0}$, one is absolutely continuous with respect to the other or they are singular.
\end{thm}

\begin{thm}
For any two elements in $\tilde{\mathcal{C}_0}$, one is weakly absolutely continuous with respect to the other or they are strongly singular.
\end{thm}

Remark that by Theorem 2.5 for a pair $\mu, \nu$ in $\tilde{\mathcal{C}_0}$, there are points $\eta = (\omega, t)$ and $\eta^{\prime} = (\omega^{\prime}, t^{\prime})$ in $\Omega^*$ such that $\mu = \nu_{\eta}$ and $\nu = \nu_{\eta^{\prime}}$.
For the sake of simplicity, in the following proofs we shall prove only in the case of $t = t^{\prime} = 0$, i.e., $\mu = \nu_{\omega}, \nu = \nu_{\omega^{\prime}}$. It is not hard to modify the proofs so that it works for arbitrarily $t, t^{\prime} \in [0,1]$.

Recall that for any set $A \in \mathcal{P}(\mathbb{N})$, $A^* = \overline{A} \cap \mathbb{N}^*$ is a clopen subset of $\mathbb{N}^*$, where $\overline{A}$ denotes the closure of $A$ in $\beta\mathbb{N}$, and these subsets form a topological basis for $\mathbb{N}^*$.

We begin with absolute continuity.
\begin{thm}
$\nu_{\eta^{\prime}} \ll \nu_{\eta}$ if and only if $\eta^{\prime} \in o_-(\eta)$.
\end{thm}

\begin{proof}
(Sufficiency) Let us $\omega^{\prime} = \tau^{-m} \omega, m \ge 0$. Fix any positive number $\delta > 0$ and let $A$ be a set with $\nu_{\omega}(A) < \delta$. We take $X \in \omega$ such that
\[\frac{|A \cap 2^n|}{2^n} < 2\delta \]
whenever $n \in X$. Remark that $\tau^{-m}X = \{n-m : n \in X \} \in \omega^{\prime}$. We have that for any $n \in X$,
\[\frac{|A \cap 2^{n-m}|}{2^{n-m}} \le 2^m \cdot \frac{|A \cap 2^{n-m}|}{2^n} \le 2^m \cdot \frac{|A \cap 2^n|}{2^n} < 2^{m+1} \cdot \delta \]
then
\[\nu_{\omega^{\prime}}(A) = \omega^{\prime}-\lim_n \frac{|A \cap 2^n|}{2^n} \le \limsup_{n \in \tau^{-m}X}  \frac{|A \cap 2^n|}{2^n} \le 2^{m+1} \cdot \delta.\]
Hence for any given $\varepsilon > 0$, put $\delta < \frac{\varepsilon}{2^{m+1}}$, then for any $A \in \mathcal{P}(\mathbb{N})$ we have 
\[\nu_{\omega}(A) < \delta \Longrightarrow \nu_{\omega^{\prime}}(A) < \varepsilon. \]
This complete the proof.

(Necessity) We shall show the contrapositive. Now we assume that $\omega^{\prime} \not\in o_-(\omega)$. Then for any positive integer m, we can take $A \in \omega, B \in \omega^{\prime}$ such that $(B \cup \tau B \cup \cdots \cup \tau^{m-1}B) \cap A = \emptyset$. Now we put $I_m = \cup_{n \in B} (2^{n-1}, 2^n]$, then we have
\[\nu_{\omega^{\prime}}(I_m) = \omega^{\prime}-\lim_n \frac{|I_m \cap 2^n|}{2^n} \ge 
\liminf_{n \in B} \frac{|I_m \cap 2^n|}{2^n} \ge \frac{1}{2}. \]
On the other hand, notice that $n \in A$ implies that $B \cap [n-m+1, n] = \emptyset$, then
\[\nu_{\omega}(I_m) = \omega-\lim_n \frac{|I_m \cap 2^n|}{2^n} \le \limsup_{n \in A} \frac{|I_m \cap 2^n|}{2^n} \le \frac{2^{n-m}}{2^n} = \frac{1}{2^m}. \]
Hence for any $\delta > 0$, Choose any positive integer m with $\frac{1}{2^m} < \delta$, we have
\[\nu_{\omega}(I_m) < \delta \quad and \quad \nu_{\omega^{\prime}}(I_m) \ge \frac{1}{2}. \]
Thus we conclude that $\nu_{\omega^{\prime}} \not\ll \nu_{\omega}$.
\end{proof}

From this theorem we can easily obtain the following results.
\begin{thm}
$\nu_{\eta} \ll \nu_{\eta^{\prime}}$ or $\nu_{\eta^{\prime}} \ll \nu_{\eta}$ if and only if $\eta^{\prime} \in o(\eta)$.
\end{thm}

\begin{thm}
If $\mu$ and $\nu$ are elements of $\tilde{\mathcal{C}_0}$ and mutually absolutely continuous, that is, $\nu \ll \mu$ and $\mu \ll \nu$, then $\mu = \nu$.
\end{thm}

\begin{proof}
Let $\mu = \nu_{\eta}$ and $\nu = \nu_{\eta^{\prime}}$ for some $\eta, \eta^{\prime} \in \Omega^*$. From the assumption of the theorem and Theorem 3.3, we can write $\eta^{\prime} = \phi^{-s} \eta$ and $\eta = \phi^{-t} \eta^{\prime}$ for some $s,t \ge 0$. Thus we have that $\eta = \phi^{-(s+t)} \eta$.
As is well known, since there are no periodic points for the flow $(\Omega^*, \{\phi^s\}_{s \in \mathbb{R}})$, we get $s+t=0$, i.e., $s = t = 0$. Hence $\eta = \eta^{\prime}$. We get the result.
\end{proof}

\begin{rem}
 Let $\eta, \eta^{\prime}$ be any two elements of $\Omega^*$, then we define a partial order $\le$ on $\Omega^*$ as follows:
\[\eta^{\prime} \le \eta \Longleftrightarrow \eta^{\prime} \in o_-(\eta) \]
Then Theorem 3.3 suggests that $\Phi$ is also an isomorphism between partially ordered sets $(\Omega^*, \le)$ and $(\tilde{\mathcal{C}_0}, \ll)$.
\end{rem}

Next we prove the result of singularity.
\begin{thm}
$\nu_{\eta}$ and $\nu_{\eta^{\prime}}$ are singular if and only if $o(\omega) \cap o(\omega^{\prime}) = \emptyset$, i.e., $\omega^{\prime} \not\in o(\omega)$.
\end{thm}

\begin{proof}
Necessity is obvious by the Theorem 3.4. Hence we shall prove sufficiency. For any positive integer $m$, take a set $B$ in $\omega^{\prime}$ such that  $\{\tau^{-(m-1)}\omega, \cdots, \tau^{-1}\omega,
\omega, \tau\omega, \cdots, \tau^{m-1}\omega\} \cap B^* = \emptyset$. Thus we have
\[\omega \notin \cup_{i=-(m-1)}^{m-1} \tau^i B^*. \]
Then take $A$ in $\omega$ such that $(\cup_{i=-(m-1)}^{m-1} \tau^i B) \cap A = \emptyset$.
Then we define $J_m = \cup_{n \in B} (2^{n-m}, 2^n], \ m \ge 1$. Hence
\[\nu_{\omega^{\prime}}(J_m) = \omega^{\prime}-\lim_n \frac{|J_m \cap 2^n|}{2^n} \ge \liminf_{n \in B} \frac{|J_m \cap 2^n|}{2^n} \ge \liminf_{n \in B} \frac{2^n - 2^{n-m}}{2^n} = 1 - \frac{1}{2^{m}}.\]
On the other hand, since if $n \in A$, then $[n-m+1, n] \cap B = \emptyset$, we have
\[\nu_{\omega}(J_m) = \omega-\lim_n \frac{|J_m \cap 2^n|}{2^n} \le \limsup_{n \in A} \frac{|J_m \cap 2^n|}{2^n} \le \frac{2^{n-m}}{2^n} = \frac{1}{2^m}.\]
Therefore for any $\varepsilon > 0$, take any positive integer $m$ with $\frac{1}{2^m} < \varepsilon$, we get
\[\nu_{\omega}(J_m) \le \varepsilon, \quad and \quad \nu_{\omega^{\prime}}({J_m}^c) \le \varepsilon.\]
Thus we conclude that $\nu_{\omega} \perp \nu_{\omega^{\prime}}$.
\end{proof}

 By Theorem 3.4 and Theorem 3.6 we obtain Theorem3.1. Next we take up weakly absolutely continuity and strong singularity.

\begin{thm}
$\nu_{\eta^{\prime}} \prec \nu_{\eta}$ if and only if $\eta^{\prime} \in \overline{o}_-({\eta})$.
\end{thm}

\begin{proof}
(Sufficiency) Remark that it suffices to show that
\[\nu_{\omega^{\prime}}(A) > 0 \Longrightarrow \nu_{\omega}(A) > 0. \]
For any $X^{\prime}$ in $\omega^{\prime}$, there exists a integer $m \ge 1$ such that $\tau^{-m} \omega \in {X^{\prime}}^*$, i.e., $\tau^m X^{\prime} \in \omega$. Now we take any $A \in \mathcal{P}(\mathbb{N})$ such that
\[\nu_{\omega^{\prime}}(A) = \omega^{\prime}-\lim_n \frac{|A \cap 2^n|}{2^n} = \delta > 0\]
and then take $X^{\prime} \in \omega^{\prime}$ such that
\[n \in X^{\prime} \Longrightarrow \frac{|A \cap 2^n|}{2^n} > \frac{\delta}{2}. \]
Then we have
\[\frac{|A \cap 2^{n+m}|}{2^{n+m}} \ge \frac{1}{2^m} \cdot \frac{|A \cap 2^{n+m}|}{2^n} \ge \frac{1}{2^m} \cdot \frac{|A \cap 2^n|}{2^n} > \frac{\delta}{2^{m+1}} \]
for any $n \in X$. Hence
\[\nu_{\omega}(A) = \omega-\lim_n \frac{|A \cap 2^n|}{2^n} \ge \liminf_{n \in \tau^m X^{\prime}} \frac{|A \cap 2^n|}{2^n} \ge \frac{\delta}{2^{m+1}} > 0.\]
This complete the proof.

(Necessity) We shall show the contrapositive. Assume that $\omega^{\prime} \not\in \overline{o}_-(\omega) = \emptyset$. Then there exists a set $X \in \omega^{\prime}$ such that $X^* \cap o_-(\omega) = \emptyset$, i.e., $\omega \not\in \cup_{i \ge 0} \tau^i X^*$. Hence for any fixed positive integer $m$, there is a set $A_m \in \omega$ with $A_m \cap (X \cup \tau X \cup \cdots \cup \tau^{m-1} X) = \emptyset$. Then we define $I = 
\cup_{n \in X} (2^{n-1}, 2^n]$. For any $n \in A_m$, since $X \cap [n-m+1, n] = \emptyset$ we have
\[\nu_{\omega}(I) = \omega-\lim_n \frac{|I \cap 2^n|}{2^n} \le \limsup_{n \in A_m} \frac{|I \cap 2^n|}{2^n} \le \frac{2^{n-m}}{2^n}=\frac{1}{2^m}. \]
Thus since $m \ge 1$ can be arbitrary, we have $\nu_{\omega}(I_m) = 0$. On the other hand,
\[\nu_{\omega^{\prime}}(I) = \omega^{\prime}-\lim_n \frac{|I \cap 2^n|}{2^n} \ge \liminf_{n \in X} \frac{|I \cap 2^n|}{2^n} \ge \frac{1}{2}. \]
Hence we have seen that $I \subseteq \mathbb{N}$ satisfies $\nu_{\omega}(I) = 0$ and $\nu_{\omega^{\prime}}(I) > 0$. Thus
$\nu_{\omega^{\prime}} \not\prec \nu_{\omega}$. We obtain the result.
\end{proof}

\begin{rem}
 Let $\eta, \eta^{\prime}$ be any two elements of $\Omega^*$, then we define a preorder $\sqsubseteq$ on $\Omega^*$ as follows:  
\[\eta^{\prime} \sqsubseteq \eta \Longleftrightarrow \eta^{\prime} \in \overline{o}_-(\eta) \]
Then Theorem 3.7 suggests that $\Phi$ is a isomorphism between preordered sets $(\Omega^*, \sqsubseteq)$ and $(\tilde{\mathcal{C}_0}, \prec)$.
\end{rem}

Finally, we shall show the following result about strongly singularity.

\begin{thm}
$\nu_{\eta}$ and $\nu_{\eta^{\prime}}$ are strongly singular if and only if $\eta^{\prime} \notin \overline{o}_-(\eta)$ and $\eta \notin \overline{o}_-(\eta^{\prime})$.
\end{thm}

\begin{proof}
Necessity is obvious by Theorem 3.7. Hence we will show sufficiency. At first, by the assumption, we can take disjoint sets $X \in \omega$ and $X^{\prime} \in \omega^{\prime}$ such that 
$X^* \cap \overline{o}_-(\omega^{\prime}) = \emptyset, \ {X^{\prime}}^* \cap \overline{o}_-(\omega) = \emptyset$, i.e., $\cup_{i \ge 0} \tau^i X^* \not\ni \omega^{\prime}$ and $\cup_{i \ge 0} \tau^i {X^{\prime}}^* \not \ni \omega$. Now we put $X_0 = X, X_0^{\prime} = X^{\prime}$, and construct decreasing sequences $\{X_i\}_{i \ge 0}, \{{X_i}^{\prime}\}_{i \ge 0}$ of $\omega, \omega^{\prime}$ inductively such that for every $m \ge 0$,
\[(X_0 \cup \tau^{-1}X_1 \cup \cdots \cup \tau^{-m}X_m) \cap ({X_0}^{\prime} \cup \tau^{-1}{X_1}^{\prime} \cup \cdots \cup \tau^{-m}{X_m}^{\prime}) = \emptyset. \]
Assume that we have constructed to m, then we define sets $X_{m+1}, {X_{m+1}}^{\prime}$. At first for $X_{m+1}$, it is necessary and sufficient that
\[\tau^{-(m+1)}X_{m+1} \cap {X_0}^{\prime} = \emptyset, \tau^{-(m+1)}X_{m+1} \cap \tau^{-1}{X_1}^{\prime} = \emptyset, \cdots , \tau^{-(m+1)}X_{m+1} \cap \tau^{-m}{X_m}^{\prime} = \emptyset. \]
This is equivalent to the following:
\[X_{m+1} \cap \tau^{m+1}{X_0}^{\prime} = \emptyset, X_{m+1} \cap \tau^m {X_1}^{\prime} = \emptyset, \cdots, X_{m+1} \cap \tau {X_m}^{\prime} = \emptyset. \]
By the assumptions that $\omega \notin \cup_{i \ge 0} \tau^i {X^{\prime}}^*$ and ${X_0}^{\prime} \supseteq {X_1}^{\prime} \supseteq \cdots \supseteq {X_m}^{\prime}$, we can take a set $X_{m+1} \in \omega$ with $X_{m+1} \subseteq X_m$ satisfying this condition. In the same way, we can take a set ${X_{m+1}}^{\prime} \in \omega^{\prime}$ with ${X_{m+1}}^{\prime}
\subseteq {X_m}^{\prime}$ such that
\[\tau^{-(m+1)}{X_{m+1}}^{\prime} \cap X_0 = \emptyset, \tau^{-(m+1)}{X_{m+1}}^{\prime}  \cap \tau^{-1} X_1 = \emptyset, \cdots , \tau^{-(m+1)}{X_{m+1}}^{\prime} \cap \tau^{-m} X_m = \emptyset. \]
Then obviously $X_{m+1} \cap {X_{m+1}}^{\prime} = \emptyset$ and we have that
\[(X_0 \cup \tau^{-1}X_1 \cup \cdots \cup \tau^{-{(m+1)}}X_{m+1}) \cap ({X_0}^{\prime} \cup \tau^{-1}{X_1}^{\prime} \cup \cdots \cup \tau^{-{(m+1)}}{X_{m+1}}^{\prime}) = \emptyset. \]
Now for these $\{X_i\}_{i \ge 0}, \{{X_i}^{\prime}\}_{i \ge 0}$, define
\[Y = \cup_{i=0}^{\infty} \tau^{-i}X_i, \quad Y^{\prime} = \cup_{i=0}^{\infty} \tau^{-i}{X_i}^{\prime}\]
then $Y \cap Y^{\prime} = \emptyset$, and define
\[I = \cup_{n \in Y} (2^{n-1}, 2^n], \quad J = \cup_{n \in Y^{\prime}} (2^{n-1}, 2^n]. \]
Since $Y \cap Y^{\prime} = \emptyset$, then $I \cap J = \emptyset$. For every $m \ge 1$, $n \in X_{m-1}$ implies that $n-m+1, \cdots, n \in Y$. Hence
\[n \in X_{m-1} \Longrightarrow (2^{n-m}, 2^n] \subseteq I. \]
Then we have 
\[\nu_{\omega}(I) = \omega-\lim_n \frac{|I \cap 2^n|}{2^n} \ge \liminf_{n \in X_{m-1}} \frac{|I \cap 2^n|}{2^n} \ge \frac{2^n - 2^{n-m}}{2^n}=1-\frac{1}{2^m}. \]
Since $m \ge 1$ is arbitrary, we get $\nu_{\omega}(I) =1$. In a similar way, we also get $\nu_{\omega^{\prime}}(J) = 1$. Thus we obtain $\nu_{\omega} \ \rotatebox{90}{$\vDash$} \ \nu_{{\omega}^{\prime}}$.
\end{proof}

Combining Theorem 3.7 and Theorem 3.8, we obtain Theorem 3.2. 
We remark that Theorem 3.7 is equivalent to the following.

\begin{thm}
$\nu_{\eta}$ and $\nu_{\eta^{\prime}}$ are strongly singular if and only if $\overline{o}_-(\eta) \cap \overline{o}_-(\eta^{\prime}) = \emptyset$.
\end{thm}

\begin{proof}
Sufficiency is obvious. On the other hand, notice that necessity is apparently stronger than the claim in Theorem 3.7. Assume that $\nu_{\omega} \ \rotatebox{90}{$\vDash$} \ \nu_{\omega^{\prime}}$ and $\overline{o}_-(\omega) \cap \overline{o}_-(\omega^{\prime}) \not= \emptyset$. Then there exists a $\omega^{\prime{\prime}} \in \mathbb{N}^*$ such that $\omega^{\prime{\prime}} \in \overline{o}_-(\omega)$ and $\omega^{\prime{\prime}} \in \overline{o}_-(\omega^{\prime})$. But from Theorem 3.6 we have that $\nu_{\omega^{\prime{\prime}}} \prec \nu_{\omega}$ and $\nu_{\omega^{\prime{\prime}}} \prec \nu_{\omega^{\prime}}$. This implies that supp $\nu_{\omega^{\prime{\prime}}}$ $\subseteq$ supp $\nu_{\omega}$ $\cap$ supp $\nu_{\omega^{\prime}} \not= \emptyset$, which contradicts the assumption.
\end{proof}

\section{Applications}
In this section we shall give some consequences of the results of the previous section. As we have seen in section 3, properties of the orbit of $\eta$ for $\{\phi^s\}_{s \in \mathbb{R}}$ give rise to properties of the density measure $\nu_{\eta}$ in $\tilde{\mathcal{C}_0}$. In what follows, we particularly pay attention to recurrence of the orbit of a point $\eta$ in $\Omega^*$, and see how it affects properties of the density measure $\nu_{\eta}$. 

 We denote by $\mathcal{D}$ all the points $\eta$ in $\Omega^*$ which orbit $o_-(\eta)$ is not recurrent. This means that there exist a neighborhood $U$ of $\eta$ and a real number $L > 0$ such that $\phi^{-s} \eta$ does not enter $U$ for every $s > L$.

 We denote by $\mathcal{A}$ the set of all almost periodic points for the flow $(\Omega^*, \{\phi^s\}_{s \in \mathbb{R}})$. Recall that we say that a point $\eta$ in $\Omega^*$ is almost periodic if the orbit closure $\overline{o}(\eta)$ is a minimal closed invariant set.

\begin{thm}
(1) If $\eta \in \mathcal{D}$, then supp $\nu_{\eta^{\prime}}$ $\subsetneq$ supp $\nu_{\eta}$ for any $\eta^{\prime} \in \overline{o}_-(\eta) \setminus \{\eta\}$.

(2) If $\eta \notin \mathcal{D}$ then supp~$\nu_{\eta^{\prime}}$ $\subseteq$ supp~$\nu_{\eta}$ for any $\eta^{\prime} \in \overline{o}(\eta)$. In particular, supp~$\nu_{\eta^{\prime}}$ = supp~$\nu_{\eta}$ for any $\eta^{\prime} \in o(\eta)$. 

(3) If $\eta \in \mathcal{A}$, then supp $\nu_{\eta^{\prime}}$ = supp $\nu_{\eta}$ for any $\eta^{\prime} \in \overline{o}(\eta)$.
\end{thm}

\begin{proof}
(1) By Theorem 3.7 for any $\eta^{\prime} \in \overline{o}_-(\eta)$, supp $\nu_{\eta^{\prime}}$ $\subseteq$ supp $\nu_{\eta}$. Also $\eta \notin \overline{o}_-(\eta^{\prime})$ since $\eta$ is 
not recurrent. Then again by Theorem 3.7 supp $\nu_{\eta}$ $\subsetneq$ supp $\nu_{\eta^{\prime}}$. Hence supp $\nu_{\eta^{\prime}}$ $\subsetneq$ supp $\nu_{\eta}$.

(2) Observe that if $\eta \in \mathcal{D}$ then $\overline{o}_-(\eta) = \overline{o}(\eta)$. Then first half of the claim follows immediately. Without loss of generality, we can assume that $\eta^{\prime} \in o_-(\eta)$. By Theorem 3.4 we get that supp $\nu_{\eta^{\prime}}$ $\subseteq$ supp $\nu_{\eta}$. On the other hand,
$\eta \in \overline{o}_-(\eta^{\prime})$ since $\eta \notin \mathcal{D}$. That is, by Theorem 3.7, supp $\nu_{\eta}$ $\subseteq$ supp $\nu_{\eta^{\prime}}$. Hence supp $\nu_{\eta}$ = supp $\nu_{\eta^{\prime}}$.

(3) Notice that for any $\eta^{\prime} \in \overline{o}(\eta) = \overline{o}_-(\eta)$, the orbit  $o_-(\eta^{\prime})$ is dense in $\overline{o}(\eta)$, that is, $\eta \in \overline{o}_-(\eta^{\prime})$ i.e., supp $\nu_{\eta}$ $\subseteq$ supp $\nu_{\eta^{\prime}}$. Then we get the result immediately.
\end{proof}

Now we shall deal with the additive property of elements of $\mathcal{C}_0$. This notion has been considered in the context of the completeness of the $L_p$-spaces over finitely additive measures (See [1]), Namely, $L_p(\mu)$ is complete if and only if $\mu$ has the additive property. The additive property for density measures was studied in [4], [6] and [8]. The definition of it is as follows:
We say that a measure $\mu$ has the additive property if for any increasing sequence $A_1 \subset A_2 \subset \cdots \subset A_k \subset \cdots$ of $\mathcal{P}(\mathbb{N})$, there exists a set $B \subseteq \mathbb{N}$ such that

\quad (1) $\mu(B) = \lim_k \mu(A_k)$,

\quad (2) $\mu(A_k \setminus B) = 0 \quad for \ every \ k \in \mathbb{N}$. 
\medskip

 The additive property of $\mu$ can be characterized by extending measure $\hat{\mu}$. As we have seen in section 2, we extend a measure space $(\mathbb{N}, \mathcal{P}(\mathbb{N}), \mu)$ to $(\beta\mathbb{N},
\mathcal{B}(\beta\mathbb{N}), \hat{\mu})$. We will need the following formulation of the additive property, which is a version of [2, Theorem 2].

\begin{thm}[{[1, Theorem 4.2]}]
A measure $\mu$ has the additive property if and only if $\hat{\mu}(U) = \hat{\mu}(\overline{U})$ for every open sets $U$ of supp $\mu$, where $\overline{U}$ denotes the closure of $U$ in supp $\mu$.
\end{thm}

From this we can easily obtain the following result (See for instance [1, Proposition 4.4]).
\begin{thm}
A measure $\mu$ has the additive property if and only if $\hat{\mu}(E) = 0$ for all nowhere dense Borel sets E in supp $\mu$.
\end{thm}

It is noted that several other conditions which are equivalent to the additive property are known. See [1] for further details.

In addition we will need the following two theorems in connection with results in section 3 (See [1, Proposition 8.1, 8.6, 8.9, and 8.10]).

\begin{thm}
Let $\mu$, $\nu$ be measures such that $\nu \ll \mu$ and $\mu$ has the additive property. Then $\nu$ also has the additive property.
\end{thm}

\begin{thm}
For any singular measures $\mu$, $\nu$, $\mu + \nu$ has the additive property if and only if both $\mu$ and $\nu$ have the additive property and $\mu$ and $\nu$ are strongly singular.
\end{thm}

We have proved in [6, Theorem 4.1] the following result.
\begin{thm}
$\nu_{\eta}$ has the additive property if and only if $\eta \in \mathcal{D}$.
\end{thm}

 In particular, the necessity is essentially due to [4, Theorem 6]. Here we shall give a different and simple proof of it as an application of results in section 3.

\begin{thm}
If $\eta \notin \mathcal{D}$ then $\nu_{\eta}$ does not have the additive property.
\end{thm}

\begin{proof}
 Assume that $\eta \notin \mathcal{D}$ and $\nu_{\eta}$ has the additive property. Then $\hat{\nu}_{\eta}$ vanishes on the nowhere dense Borel subsets of supp $\nu_{\eta}$. We take any $\eta^{\prime} \in o_{-}(\eta) \setminus \{\eta\}$. Since $\nu_{\eta^{\prime}} \ll \nu_{\eta}$ by Theorem 3.3, $\nu_{\eta^{\prime}}$ also has the additive property by Theorem 4.4 and supp $\nu_{\eta^{\prime}}$ = supp $\nu_{\eta}$ by Theorem 4.1(2). Hence $\hat{\nu}_{\eta^{\prime}}$ also vanishes on the nowhere dense Borel sets of supp  $\nu_{\eta^{\prime}}$, i.e., supp $\nu_{\eta}$. It follows that $\hat{\nu}_{\eta}$ and $\hat{\nu}_{\eta^{\prime}}$ are mutually absolutely continuous. Thus by Theorem 3.5 $\eta = \eta^{\prime}$, which contradicts the assumption.
\end{proof}

From Theorem 3.3 and Theorem 4.6 we get the following result.

\begin{thm}
Let $\eta \in \mathcal{D}$ and $\mu$ be a countably additive probability measure on the space $\mathbb{R}_+ = [0, \infty)$ of nonnegative real numbers. Then the density measure $\nu$ in $\mathcal{C}_0$ defined by
\[\nu(A) = \int_0^{\infty} \nu_{\phi^{-s}\eta}(A) d\mu(s), \quad A \in \mathcal{P}(\mathbb{N}) \]
has the additive property.
\end{thm}

\begin{proof}
Let $\{A_i\}_{i = 1}^{\infty}$ be an increasing sequence of $\mathcal{P}(\mathbb{N})$.
By Theorem 4.6 $\nu_{\eta}$ has the additive property. Then there is a set $B \in \mathcal{P}(\mathbb{N})$ such that $\lim_i \nu_{\eta}(A_i) = \nu_{\eta}(B)$ and $\nu_{\eta}(A_i \setminus B) = 0$ for every $i \ge 1$. Notice that the latter condition  yields that $\overline{A_i} \cap$ supp~$\nu_{\eta}$ $\subseteq$ $\overline{B} \cap$ supp~$\nu_{\eta}$ for each $i \ge 1$, then we have that 
\begin{align}
\lim_i \nu_{\eta}(A_i) = \nu_{\eta}(B) &\Longleftrightarrow \hat{\nu}_{\eta}(\cup_{i=1}^{\infty} \overline{A}_i) = \hat{\nu}_{\eta}(\overline{B}) \notag \\
&\Longleftrightarrow \hat{\nu}_{\eta}(\overline{B} \setminus (\cup_{i=1}^{\infty} \overline{A}_i)) = 0. \notag
\end{align}
For any $\eta^{\prime} \in o_-(\eta)$ since $\nu_{\eta^{\prime}} \ll \nu_{\eta}$, it follows that
\[\hat{\nu}_{\eta^{\prime}}(\overline{B} \setminus (\cup_{i=1}^{\infty} \overline{A}_i)) = 0 
\Longleftrightarrow \lim_i \nu_{\eta^{\prime}}(A_i) = \nu_{\eta^{\prime}}(B). \]
Also $\nu_{\eta^{\prime}}(A_i \setminus B) = 0$ holds for each $i \ge 1$. Then it follows that
\[\nu(A_i \setminus B) = \int_0^{\infty} \nu_{\phi^{-s}\eta}(A_i \setminus B) d\mu(s) = 0 \]
for each $i \ge 1$. Also we have that
\begin{align}
\nu(B) &= \int_0^{\infty} \nu_{\phi^{-s}\eta}(B) d\mu(s) \notag \\
&= \int_0^{\infty} \lim_i \nu_{\phi^{-s}\eta}(A_i) d\mu(s) \notag \\
&=\lim_i \int_0^{\infty} \nu_{\phi^{-s}\eta}(A_i) d\mu(s) = \lim_i \nu(A_i). \notag 
\end{align}
Hence $\nu$ has the additive property.
\end{proof}

Next from Theorem 3.9, Theorem 4.5 and Theorem 4.6 we get the following result.

\begin{thm}
For any mutually singular finite number of density measures $\nu_{\eta_1}, \nu_{\eta_2}, \cdots, \nu_{\eta_m}$, their finite convex combination
\[\nu = \sum_{i=1}^m c_i \nu_{\eta_i} \]
has the additive property if and only if $\eta_i \in \mathcal{D}, i=1,2, \cdots, m$ and orbit closures $\overline{o}_-(\eta_i), i=1,2, \cdots, m$, are pairwise disjoint.
\end{thm}

Combining Theorem 4.8 and Theorem 4.9, we have that

\begin{thm}
Take a finite number of points $\eta_1, \eta_2, \cdots, \eta_m$ in $\mathcal{D}$ such that orbit closures $\overline{o}_-(\eta_1), \overline{o}_-(\eta_2), \cdots, \overline{o}_-(\eta_m)$ are pairwise disjoint, and let $\mu_i, i=1,2, \cdots, m$, be countably additive probability measures on $\mathbb{R}_+ = [0, \infty)$. Then the density measure $\nu$ in $\mathcal{C}_0$ defined by
\[\nu = \sum_{i=1}^m c_i \int_0^{\infty}  \nu_{\varphi^{-s}\eta_i} d\mu_i(s) \]
has the additive property, where $0 \le c_i \le 1, i=1,2, \cdots, m$ with $\sum_{i=1}^m c_i =1$.
\end{thm}

\end{document}